\documentclass[a4paper, 11pt]{article}

\usepackage[utf8]{inputenc}
\usepackage[T1]{fontenc}
\usepackage[a4paper, margin=2.35cm]{geometry}
\usepackage{needspace}
\usepackage{cite}
\usepackage{textcomp} 
\usepackage{libertine} %
\usepackage{inconsolata} %
\usepackage{amsmath, amsthm, amssymb, mathtools}
\usepackage{graphicx}
\usepackage{enumerate}
\usepackage{authblk}
\usepackage[textsize=small, textwidth=2cm, color=yellow]{todonotes}
\usepackage{thmtools}
\usepackage{thm-restate}
\usepackage[colorlinks=true, linkcolor =blue!80, citecolor=orange!70!black]{hyperref}
\usepackage{float}
\usepackage[capitalize,noabbrev]{cleveref}
\usepackage{comment}
\usepackage{tikz}
\usetikzlibrary{calc,decorations,decorations.pathmorphing}
\usepackage{sidecap, caption}

\renewenvironment{abstract}
{\small\vspace{-1em}
\begin{center}
\bfseries\abstractname\vspace{-.5em}\vspace{0pt}
\end{center}
\list{}{
\setlength{\leftmargin}{0.6in}%
\setlength{\rightmargin}{\leftmargin}}%
\item\relax}
{\endlist}

\newcommand{\ev}[1]{#1}  %
\newcommand{\cv}[1]{}  %
   \newcommand{\sv}[1]{}  %
 \newcommand{\lv}[1]{#1}  %
\usepackage{etoolbox}
\newcommand{\appendixText}{}

\declaretheorem[name=Theorem, numberwithin=section]{theorem}
\declaretheorem[name=Lemma, sibling=theorem]{lemma}

\declaretheorem[name=Corollary, sibling=theorem]{corollary}

\declaretheorem[name=Claim, sibling=theorem]{claim}

\declaretheorem[name=Question, style=remark, sibling=theorem]{question}
\def\cqedsymbol{\ifmmode$\lrcorner$\else{\unskip\nobreak\hfil
\penalty50\hskip1em\null\nobreak\hfil$\lrcorner$
\parfillskip=0pt\finalhyphendemerits=0\endgraf}\fi}

\DeclareMathOperator{\pe}{\mathbf{pe}}

\DeclareMathOperator{\Oh}{\mathcal{O}}

\let\le\leqslant
\let\ge\geqslant
\let\leq\leqslant
\let\geq\geqslant

\title{Path Eccentricity and Forbidden Induced Subgraphs\thanks{An extended abstract of this paper has been accepted to Eurocomb 2025~\cite{OurEC}.}}

\author[1]{Sylwia Cichacz\thanks{Supported by the AGH University of Krakow under grant no.~16.16.420.054, funded by the Polish Ministry of Science and Higher Education.}}
\author[2]{Claire Hilaire\thanks{Supported in part by the Slovenian Research and Innovation Agency (research projects J1-4008 and N1-0370).}}
  \author[3]{Tomáš Masařík\thanks{Supported by the Polish National Science Centre SONATA-17 grant number 2021/43/D/ST6/03312.}}
  \author[3]{\\Jana Masaříková\thanks{Supported by the Polish National Science Centre Preludium research project no.\ 2022/45/N/ST6/04232.}}
   \author[2]{Martin Milanič\thanks{Supported in part by the Slovenian Research and Innovation Agency (I0-0035, research program P1-0285 and research projects J1-3003, J1-4008, J1-4084, J1-60012, and N1-0370) and by the research program CogniCom (0013103) at the University of Primorska.}}
\affil[1]{Faculty of Applied Mathematics, AGH University of Krakow, Poland}
\affil[2]{FAMNIT and IAM, University of Primorska, Koper, Slovenia}
\affil[3]{University of Warsaw, Institute of Informatics, Poland}

\date{}

\begin{document}

\maketitle

\begin{abstract}
The path eccentricity of a connected graph $G$ is the minimum integer $k$ such that $G$ has a path such that every vertex is at distance at most $k$ from the path.
A result of Duffus, Jacobson, and Gould from 1981 states that every connected $\{\text{claw}, \text{net}\}$-free graph $G$ has a Hamiltonian path, that is, $G$ has path eccentricity~$0$.
Several more recent works identified various classes of connected graphs with path eccentricity at most $1$, or, equivalently, graphs having a spanning caterpillar, including connected $P_5$-free graphs, AT-free graphs, and biconvex graphs.
Generalizing all these results, we apply the work on structural distance domination of Bacs\'o and Tuza [Discrete Math., 2012] and characterize, for every positive integer $k$, graphs such that every connected induced subgraph has path eccentricity less than $k$.
More specifically, we show that every connected $\{S_{k}, T_{k}\}$-free graph has a path eccentricity less than $k$, where $S_k$ and $T_k$ are two specific graphs of path eccentricity $k$ (a subdivided claw and the line graph of such a graph).
As a consequence, every connected $H$-free graph has path eccentricity less than $k$ if and only if $H$ is an induced subgraph of $3P_{k}$ or $P_{2k+1} + P_{k-1}$. 
\ev{For such cases, we also provide a robust polynomial-time algorithm that finds a path witnessing the upper bound on the path eccentricity.}
Our main result also answers an open question of Bastide, Hilaire, and Robinson [Discrete Math., 2025].
\end{abstract}

\section{Introduction}

A \emph{Hamiltonian path} in a graph $G$ is a path in $G$ whose vertex set is $V(G)$; a Hamiltonian cycle is defined similarly. 
The problems of determining whether a given graph has a Hamiltonian path or a Hamiltonian cycle are both \textsf{NP}-complete~\cite{MR519066,MR378476}, which motivates the search for sufficient conditions for their existence (see, e.g., the survey~\cite{MR3143857}).
One of the early results in the area is the following result of Duffus, Jacobson, and Gould~\cite{Duffus1981} (see also~\cite{Shepherd1991,Brandstadt2003,Kelmans2006}).

\begin{theorem}\label{thm:clawnetfree}
Every connected $\{$claw, net$\}$-free graph has a Hamiltonian path.
\end{theorem}

Here, a \emph{claw} is the complete bipartite graph $K_{1,3}$, the \emph{net} is the graph obtained by appending a pendant edge to each vertex of the complete graph $K_3$, and, given a set $\mathcal{F}$ of graphs, a graph $G$ is said to be \emph{$\mathcal{F}$-free} if $G$ does not have an induced subgraph isomorphic to any member of $\mathcal{F}$.
If  $\mathcal{F}= \{F\}$ for some graph $F$, we also say that $G$ is \emph{$F$-free}.

The result of \Cref{thm:clawnetfree} is tight, since the claw and the net do not have a Hamiltonian path.
Let us remark that, in general, if $G$ is a graph that has a Hamiltonian path and $H$ is a connected induced subgraph of $G$, then $H$ may fail to have a Hamiltonian path.
For example, if $H$ is any $n$-vertex graph and $G$ is a graph obtained from $H$ by adding to it $n$ new vertices fully adjacent to the vertices of $H$, then $G$ has a Hamiltonian path, even though $H$ may fail to have one.
In particular, this construction could be applied to the claw and the net. 

If we instead require that all connected induced subgraphs have a Hamiltonian path, \Cref{thm:clawnetfree} implies the following characterization.

\begin{corollary}\label{cor:clawnetfree}
For every graph $G$, the following statements are equivalent:
\begin{itemize} 
\item Every connected induced subgraph of $G$ has a Hamiltonian path.
\item $G$ is $\{$claw, net$\}$-free.
\end{itemize}
\end{corollary}

In 2023, G\'omez and Guti\'errez~\cite{Gomez2023} generalized the concept of Hamiltonian paths by defining the notion of path eccentricity of graphs.
Given a connected graph $G$ and a path $P$ in $G$, the \emph{eccentricity} of $P$ is the maximum distance from a vertex in $G$ to a vertex in $P$.
Hence, in particular, a path $P$ has eccentricity $0$ if and only if $P$ is a Hamiltonian path, and $P$ has eccentricity at most $1$ if and only if $P$ is \emph{dominating}, that is, every vertex not in the path has a neighbor on the path.
More generally, given an integer $k\ge 0$, a path $P$ is said to be \emph{$k$-dominating} if it has eccentricity at most $k$.
The \emph{path eccentricity} of a connected graph $G$ is denoted by $\pe(G)$ and defined as the minimum eccentricity of a path in $G$.
In particular, $\pe(G)=0$ if and only if $G$ has a Hamiltonian path, and $\pe(G)\le 1$ if and only if $G$ has a dominating path.
It is not difficult to see that a graph $G$ has a dominating path if and only if $G$ has a \emph{spanning caterpillar}, that is, a spanning subgraph that is a caterpillar (a tree possessing a dominating path).
Hence, understanding graphs that admit dominating paths is motivated by their applications to graph burning and a variant of the cops and robber game.
More precisely:
\begin{itemize}
    \item The burning graph conjecture (see Bonato, Janssen, and Roshanbin~\cite{Bonato2016}) is known to be true for graphs admitting a spanning caterpillar.
    \item If a graph $G$ admits a spanning caterpillar $T$, then the minimum number of cops sufficient for the cop player to win on the graph $G$ against the robber moving at the speed of $s$ is at most $ps$ if $G$ is a subgraph of the graph obtained from $T$ by adding an edge between each pair of vertices that are at distance at most $p$ in $T$ (see Fomin, Golovach, Kratochvíl, Nisse, and Suchan \cite[Lemma~4]{Fomin2010}).
\end{itemize}

Since adding a universal vertex to any connected graph $H$ results in a graph $G$ with path eccentricity at most $1$, path eccentricity is not monotone under vertex deletion.
Nevertheless, and in line with the above discussion preceding \Cref{cor:clawnetfree}, which characterizes graphs each connected induced subgraph of which has path eccentricity $0$, analogous questions can be addressed for higher values of path eccentricity.

\begin{question}
For a positive integer $k$, what are the graphs $G$ such that each connected induced subgraph of $G$ has path eccentricity at most $k$?
\end{question}

For each fixed $k\ge 1$, there are two obstructions to path eccentricity less than $k$ that can be obtained from the claw and the net, as follows.
Given a positive integer $k$, let us denote by $S_k$ the \emph{$k$-subdivided claw}, that is, the graph obtained from the claw by replacing each of its edges with a path of length $k$, where the length of a path is the number of its edges (in particular, $S_1$ is the claw).
Furthermore, we denote by $T_k$ the \emph{$k$-subdivided net}, that is, the graph obtained by appending a path of length $k$ to each vertex of the complete graph $K_3$ (in particular, $T_1$ is the net).
\lv{See \cref{fig:S2T2} for an example.

\begin{figure}[h!]
    \centering
    \includegraphics[width=0.6\linewidth]{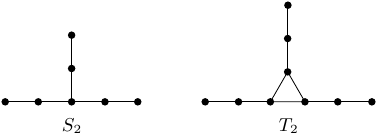}
    \caption{Visualization of graphs $S_2$ and $T_2$.}
    \label{fig:S2T2}
\end{figure}
}

It is easy to observe that \hbox{$\pe(S_k) = \pe(T_k) = k$} for each positive integer $k$.
Consequently, if $G$ is a graph such that each connected induced subgraph of $G$ has path eccentricity less than $k$, then $G$ is $\{S_{k},T_{k}\}$-free.
As the main result of this paper, we show that this necessary condition is in fact also sufficient. 
We thus obtain the following unifying theorem, valid for all positive integers $k$ (and thus capturing also \Cref{cor:clawnetfree}, which corresponds to the case $k = 1$).

\begin{restatable}{theorem}{mainThm}
  \label{thm:hereditary-path-eccentricity}
For every integer $k\ge 1$ and every graph $G$, the following statements are equivalent:
\begin{itemize}
  \item Every connected induced subgraph $H$ of $G$ satisfies $\pe(H)< k$.
  \item $G$ is $\{S_{k},T_{k}\}$-free.
\end{itemize}
\end{restatable}

The case $k = 2$ of \Cref{thm:hereditary-path-eccentricity}  characterizes graphs every connected induced subgraph of which has a dominating path, and implies the following.

\begin{corollary}\label{cor:k=1}
Every connected $\{S_2,T_2\}$-free graph has a dominating path.
\end{corollary}

\Cref{cor:k=1} unifies and generalizes several results from the literature regarding sufficient conditions for the existence of a dominating path:
\begin{sloppypar}
\begin{enumerate}
    \item One of the oldest results along these lines is the result of Bacs\'o and Tuza in 1990 (see~\cite{Bacso1990}) stating that every connected $P_5$-free graph has a dominating clique or a dominating $P_3$ (we denote by $P_k$ the $k$-vertex path).
It follows that every connected $P_5$-free graph has a dominating path.
\item A graph $G$ is said to be \emph{AT-free} if it does not have an independent set $I$ such that $|I| = 3$ and for every vertex $v\in I$, the remaining two vertices in $I$ are in the same component of $G-N[v]$.
Every connected AT-free graph has a dominating path (see Corneil, Olariu, and Stewart~\cite{Corneil1997}, as well as~\cite{Corneil1995}).
\item A graph $G$ is \emph{biconvex} if it is bipartite, with parts $A$ and $B$ that can each be linearly ordered so that for each vertex $v$ of $G$, the neighborhood of $v$ in the part not containing $v$ forms a consecutive segment of vertices with respect to the linear ordering.
G\'omez and Guti\'errez~\cite{Gomez2023} and, independently Antony, Das, Gosavi, Jacob, and Kulamarva~\cite{Antony2024}, proved that every connected biconvex graph has a dominating path (or, equivalently, a spanning caterpillar).
\end{enumerate}
\end{sloppypar}
Indeed, each of the previously stated cases corresponds to a subclass of the class of $\{S_2,T_2\}$-free graphs, hence, \Cref{cor:k=1} implies all these results.

Further consequences of \cref{thm:hereditary-path-eccentricity} are related to a recent work of Bastide, Hilaire, and Robinson~\cite{BHR25}.
The authors explored a relation between the path eccentricity of a graph and a certain property of its adjacency matrix (or equivalently, on the neighborhoods of its vertices). 
They introduced the \emph{partially augmented consecutive ones property} (denoted \emph{*-C1P} for short): a graph has the *-C1P if there exists an ordering of its vertices sending the open or closed neighborhood of every vertex to a consecutive set.
The authors showed that graphs with the *-C1P have path eccentricity at most 2, and conjectured that they actually have path eccentricity at most 1.
We show that \cref{thm:hereditary-path-eccentricity} resolves the conjecture.

\begin{restatable}{theorem}{CoPthm}
  \label{th:C1P}
     If a graph $G$ has the *-C1P, then $ \pe(G)\leq 1$. 
\end{restatable}

Another result from~\cite{BHR25} is related to the following generalization of AT-freeness.
A graph $G$ is said to be \emph{$k$-AT-free} if $G$ does not have an independent set $I$ such that $|I| = 3$ and for every vertex $v\in I$, the remaining two vertices in $I$ are in the same component of the graph $G-\{u\colon u$ is at distance at most $k$ of $v\}$.
In particular, a graph is AT-free if and only if it is $1$-AT-free.
In~\cite{BHR25}, the authors proved that for every $k\geq 1$, if a graph $G$ is $k$-AT-free, then $\pe(G)\leq k$.
This result follows from \cref{thm:hereditary-path-eccentricity}.
Indeed, it is easy to check that for every $k\ge 2$, the graphs $S_{k}$ and $T_{k-1}$ are not $(k-1)$-AT-free. Thus, if a graph $G$ is $(k-1)$-AT-free, then $G$ is  $\{S_{k},T_{k}\}$-free, which implies by \cref{thm:hereditary-path-eccentricity} that $\pe(G)< k$. 

\bigskip

For the case where a single induced subgraph is excluded, \Cref{thm:hereditary-path-eccentricity} leads to the following characterization.
The symbol $+$ denotes the disjoint union of graphs, and for an integer $k\ge 0$ and a graph $H$, we denote by $kH$ the disjoint union of $k$ copies of $H$.

\begin{restatable}{corollary}{mainCor}\label{cor:H-freeGeneral}
Let $H$ be a graph and $k\ge 1$.
Then, the following statements are equivalent:
\begin{itemize}
  \item Every connected $H$-free graph $G$ has $\pe(G)< k$.
  \item  $H$ is an induced subgraph of $3P_{k}$ or $P_{2k+1}+P_{k-1}$.
\end{itemize}
\end{restatable}

In the case of $3P_2$-free graphs, we obtain a much stronger property. 
We can find a path with eccentricity at most $1$ that is a longest path in the graph.
While the proof of this result implies that we can transform a longest path into a longest path that is dominating in polynomial time, it is \textsf{NP}-hard to find a longest path in $3P_2$-free graphs.
In fact, determining whether a given graph has a Hamiltonian path is \textsf{NP}-complete in the class of split graphs \cite{MR1405031}, a subclass of $2P_2$-free graphs.

\begin{restatable}{theorem}{longestThm}
  \label{thm:longest-path}
Every connected $3P_{2}$-free graph $G$ has a longest path that is dominating.
\cv{
Moreover, we can find a (not necessarily longest) dominating path in $G$ in time $\mathcal{O}(n^2(n+m))$, where $n$ is the number of vertices and $m$ is the number of edges in $G$.
}
\end{restatable}

\begin{SCfigure}[50][h!]
\centering
\includegraphics[width=0.35\textwidth]{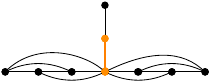}
\caption{An example of a $\{P_5,4P_2\}$-free graph no longest path of which is dominating.
The vertices depicted in orange need to belong to every dominating path; however, no longest path contains the top orange vertex.}
\label{fig:P5P1}
\end{SCfigure}
This is a very fine line, as this property is completely violated in the case of $\{S_2,T_2\}$-free graphs or even $\{P_5, 4P_2\}$-free graphs; see \cref{fig:P5P1}.
It also cannot be generalized to higher eccentricity; see for example, a $\{3P_3,P_7+P_2\}$-free graph in \cref{fig:3P3}.
Moreover, in a $3P_2$-free graph, there might exist a longest path that is not dominating, as depicted in \cref{fig:3P2}.

\begin{SCfigure}[60][h!]
\centering
\includegraphics[width=0.45\textwidth]{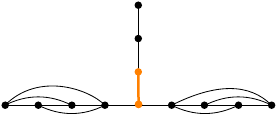}
\caption{An example of a $\{3P_3,P_7+P_2\}$-free graph no longest path of which is $2$-dominating.
The vertices depicted in orange need to belong to every $2$-dominating path; however, no longest path contains the top orange vertex.}
\label{fig:3P3}
\end{SCfigure}
\begin{SCfigure}[60][h!]
\centering
\includegraphics[width=0.22\textwidth]{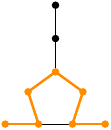}
\caption{An example of $3P_2$-free graph with one of the longest paths (depicted in orange) that is not dominating.}
  \label{fig:3P2}
\end{SCfigure}

\cv{
While \cref{thm:longest-path} actually provides us with a polynomial-time algorithm that finds a dominating path, the algorithm does not necessarily find a \textsl{longest} path that is dominating, unless we are given another longest path on the input.
}
\ev{
Note that for $3P_2$-free graphs, the algorithm does not necessarily find a \textsl{longest} path that is dominating, unless we are given another longest path on the input.
The ideas in the proof of \cref{thm:longest-path} can be generalized to the setting of \Cref{cor:H-freeGeneral}, providing us with a polynomial-time algorithm that finds a $(k-1)$-dominating path.
  The algorithm is \textsl{robust} in the sense of Spinrad and Raghavan~\cite{RaghavanS03}: it works on any input graph $G$, not necessarily $3P_k$-free or ($P_{2k+1}+P_{k-1}$)-free, and either correctly finds \textsl{both} of these induced subgraphs, or a path with eccentricity less than $k$.
Moreover, the integer $k$ can be part of the input, and the running time is independent of $k$.

To state our next result, we require a generalization of the graphs $S_k$ and $T_k$ to a broader class of graphs, represented in \cref{fig:S-T-M} (three leftmost figures) and defined as follows. 
Following the notation introduced in \cite{DDHMPT24}, we denote by $\mathcal{S}$ the class of subdivisions of the graph $K_{1,3}$. The class $\mathcal{T}$ is the class of graphs that can be obtained from three paths of length at least one by selecting one endpoint of each path and adding three edges between those endpoints to create a triangle.
The class $\mathcal{M}$ is the class of graphs $H$ that consist of a path $P$ and a vertex $a$, called the \emph{center} of $H$, such that $a$ is nonadjacent to the endpoints of $P$ and $a$ has at least two neighbors in $P$.
Given a graph $H\in \mathcal{S}\cup\mathcal{T}\cup\mathcal{M}$, the \emph{extremities} of $H$ are the vertices of degree one as well as the center of $H$ in case $H\in\mathcal{M}$. 
Observe that any $H\in \mathcal{S}\cup\mathcal{T}\cup\mathcal{M}$ has exactly three extremities.
Note in particular that for every $k\geq 1$, $S_k\in \mathcal{S}$, and $T_k\in \mathcal{T}$.
For every $k\geq 1$, we denote by $\mathcal{M}_k$ the family of graphs obtained by replacing each extremity of a graph $M\in \mathcal{M}$ by the extremity of a $P_k$.

\begin{figure}[htb]
  \centering
  \begin{tikzpicture}[scale=.8]
    \tikzset{every node/.style={draw,circle,fill=black,inner sep=0pt,minimum size=5pt,very thick}}
    \tikzstyle{vtxExtremity}=[fill=white, minimum size=6pt]
    \tikzstyle{edgePath}=[decorate,decoration={snake}]
    \tikzstyle{edgePk}=[draw=blue,dashed,very thick]
    \tikzset{decoration={amplitude=.5mm,segment length=2.25mm,post length=0mm,pre length=0mm,amplitude=.5mm}}
    \begin{scope}[]
      \node[vtxExtremity] (b) at (2,0) {};
      \node[vtxExtremity] (c) at (0,2) {};
      \node[vtxExtremity] (d) at (-2,0) {};
      \node (a) at (0,0) {};
      \draw[edgePath] (d) to (b);
      \draw[edgePath] (a) to (c);
      \node[rectangle,draw=none,fill=none] at (0,-.5) {$\strut G \in \mathcal{S}$};
    \end{scope}
    \begin{scope}[xshift=5cm]
      \node[vtxExtremity] (b) at (2,0) {};
      \node[vtxExtremity] (c) at (0,2) {};
      \node[vtxExtremity] (d) at (-2,0) {};
      \node (ac) at (0,1) {};
      \node (ab) at (0.7,0) {};
      \node (ad) at (-0.7,0) {};
      \draw (ac) to (ab) to (ad) to (ac);
      \draw[edgePath] (ab) to (b);
      \draw[edgePath] (ac) to (c);
      \draw[edgePath] (ad) to (d);
      \node[rectangle,draw=none,fill=none] at (0,-.5) {$\strut G \in \mathcal{T}$};
    \end{scope}
    \begin{scope}[xshift=10cm]
      \node[vtxExtremity] (b) at (2,0) {};
      \node[vtxExtremity] (c) at (0,2) {};
      \node[vtxExtremity] (d) at (-2,0) {};
      \node (ab) at (0.8,0) {};
      \node (ad) at (-0.8,0) {};
      \draw[edgePath] (b) to (d);
      \draw (ab) to (c) to (ad);
      \draw[dotted] (0,0) to (c) to (0.4,0)  (-0.4,0) to (c);
      \node[rectangle,draw=none,fill=none] at (0,-.5) {$\strut G \in \mathcal{M}$};
    \end{scope}
    \begin{scope}[xshift=15cm]
      \node[fill=blue] (b) at (2,0) {};
      \node[fill=blue] (c) at (0,1.5) {};
      \node[fill=blue] (d) at (-2,0) {};
      \node[fill=blue] (bb) at (2,1) {};
      \node[fill=blue] (cc) at (0,2.5) {};
      \node[fill=blue] (dd) at (-2,1) {};
      \node (ab) at (0.8,0) {};
      \node (ad) at (-0.8,0) {};
      \draw[edgePath] (b) to (d);
      \draw[edgePk] (b) to (bb);
      \draw[edgePk] (c) to (cc);
      \draw[edgePk] (d) to (dd);
      \draw (ab) to (c) to (ad);
      \draw[dotted] (0,0) to (c) to (0.4,0)  (-0.4,0) to (c);
      \node[rectangle,draw=none,fill=none] at (0,-.5) {$\strut G \in \mathcal{M}_k$};
    \end{scope}
  \end{tikzpicture}
  \caption{The three leftmost figures are representations of graphs in $\mathcal{S}$, $\mathcal{T}$, and $\mathcal{M}$, respectively from left to right, where the extremities are represented with white vertices; the rightmost figure represents a graph in $\mathcal{M}_k$, for some $k\geq 1$, where each dashed blue edge represents a $P_k$. 
  In all these figures, solid edges represent edges of the graph, wavy edges represent paths of length at least $1$, and dotted edges may be present or not.}
  \label{fig:S-T-M}
\end{figure}

We are now ready to state our next theorem.
\begin{restatable}{theorem}{thmAlg}\label{thm:3P2}
 There is an algorithm running in time $\mathcal{O}(n^2(n+m))$ that, given a connected graph $G$ with $n$ vertices and $m$ edges, and an integer $k\ge 1$, finds one of the following:
  \begin{itemize}
    \item a path with eccentricity less than $k$ in $G$, or
    \item an induced subgraph $H$ of $G$ isomorphic to either $S_k$, $T_k$, or a graph in $\mathcal{M}_k$.
  \end{itemize}  

\end{restatable}

Let us notice that it is easy to see that in every case of the second outcome, $H$ contains as an induced subgraph both $3P_{k}$ and $P_{2k+1}+P_{k-1}$, giving the aforementioned algorithmic variant of \cref{cor:H-freeGeneral}. 
Moreover, it also holds that $H$ contains either $S_k$ minus a leaf or $T_k$ minus a leaf as an induced subgraph, bringing us closer to an algorithmic version of \cref{thm:hereditary-path-eccentricity}.

}
\lv{\paragraph{Organization of the paper.}

In \cref{sec:background}, we give an overview of known results on structural domination, including a key result of Bacs\'o and Tuza~\cite{Bacso2012} that we use in the proof of our main theorem (\Cref{thm:hereditary-path-eccentricity}).
We prove \Cref{thm:hereditary-path-eccentricity,cor:H-freeGeneral} in \cref{sec:main}.
We prove \Cref{th:C1P} in \cref{sec:C1P}.
\ev{We prove \Cref{thm:longest-path,thm:3P2} in \cref{sec:additional}.}
\cv{We prove \Cref{thm:longest-path} in \cref{sec:additional}.}
We conclude the paper with some final remarks and open questions in \cref{sec:conclusion}.
}

\sv{
\bigskip 

\noindent In \cref{sec:background}, we give an overview of known results on structural domination, including a key result of Bacs\'o and Tuza~\cite{Bacso2012} that we use in the proof of our main theorem (\Cref{thm:hereditary-path-eccentricity}), which we prove in \cref{sec:main}.
We postpone the proofs of our remaining theorems to the full version of this paper (see~\cite{ourArxiv}).}

\section{Background on Structural Domination}\label{sec:background} %

\Cref{thm:hereditary-path-eccentricity} is closely related to the framework of \textsl{structural domination}.
Given a class $\mathcal{D}$ of connected graphs, let us say that a graph $G$ is \emph{hereditarily dominated} by $\mathcal{D}$ if each connected induced subgraph of $G$ contains a dominating set that induces a graph from $\mathcal{D}$. 
Bacs\'o, Michalak, and Tuza (see~\cite{Bacso2005}) gave forbidden induced subgraph characterizations of graphs hereditarily dominated by classes of complete bipartite graphs, stars, connected bipartite graphs, and complete $k$-partite graphs.
Similar results were obtained by Michalak (see~\cite{Michalak2007}) for graphs hereditarily dominated by classes of cycles and paths, paths, trees, and trees with bounded diameter.
Camby and Schaudt (see~\cite{Camby2016}) proved that for every $k\ge 4$, a graph $G$ is $P_k$-free if and only if it is hereditarily dominated by the class consisting of all connected $P_{k-2}$-free graphs along with the $k$-vertex cycle.

The most general result along these lines was obtained independently by Bacs\'o and Tuza (see~\cite{Tuza2008,Bacso2009}).
To state the result, we need one more definition.
Given a connected graph $H$, the \emph{leaf graph} of $H$ is the graph $H^{\text{\textleaf}}$ obtained from $H$ by attaching a leaf to each vertex $v\in V(H)$ that is not a cut-vertex (that is, the graph $H-v$ is connected).
By definition, these new leaves are pairwise distinct and non-adjacent. 

\begin{theorem}[Bacs\'o~\cite{Bacso2009} and Tuza~\cite{Tuza2008}]\label{thm:B09T08}
Let $\mathcal{D}$ be a class of connected graphs that contains the class of all paths, excludes at least one connected graph, and is closed under taking connected induced subgraphs.
Let $\mathcal{F}$ be the family of all connected graphs $F$ that do not belong to $\mathcal{D}$, but all proper connected induced subgraphs of $F$ belong to $\mathcal{D}$.
Let $\mathcal{F}^{\text{\textleaf}}$ be the family of all leaf graphs of graphs in $\mathcal{F}$.
Then, a graph $G$ is hereditarily dominated by $\mathcal{D}$ if and only if $G$ is $\mathcal{F}^{\text{\textleaf}}$-free.
\end{theorem}

The main result of~\cite{Tuza2008,Bacso2009} is even more general, without the restriction that $\mathcal{D}$ contains the class of all paths; but since we do not need this result, we omit the precise statement.
A generalization of \Cref{thm:B09T08} to domination at distance was proved by Bacs\'o and Tuza (see~\cite{Bacso2012}).
A \emph{distance-$k$ dominating set} in a graph $G$ is a set $D\subseteq V(G)$ such that every vertex in $G$ is at distance at most $k$ from some vertex in $D$.
If, in addition, the graph $G[D]$ is connected, then $D$ is said to be a \emph{connected distance-$k$ dominating set}.
A graph $G$ is said to be \emph{hereditarily $k$-dominated} by a class $\mathcal{D}$ of connected graphs if each connected induced subgraph of $G$ admits a distance-$k$ dominating set that induces a graph from $\mathcal{D}$. 
Given a connected graph $H$, the \emph{$k$-leaf graph} of $H$ is the graph $H_k^{\text{\textleaf}}$ obtained from $H$ by attaching a pendant path of length $k$ to each vertex $v\in V(H)$ that is not a cut-vertex.
Bacs\'o and Tuza (see~\cite{Bacso2012}) generalized \Cref{thm:B09T08} as follows.

\begin{theorem}[Bacs\'o and Tuza~\cite{Bacso2012}]\label{thm:BT12distance}
Let $\mathcal{D}$ be a class of connected graphs that contains the class of all paths, excludes at least one connected graph, and is closed under taking connected induced subgraphs.
Let $\mathcal{F}$ be the family of all connected graphs $F$ that do not belong to $\mathcal{D}$, but all proper connected induced subgraphs of $F$ belong to $\mathcal{D}$.
Let $k$ be a positive integer and let $\mathcal{F}_k^{\text{\textleaf}}$ be the family of all $k$-leaf graphs of graphs in $\mathcal{F}$.
Then, a graph $G$ is hereditarily $k$-dominated by $\mathcal{D}$ if and only if $G$ is $\mathcal{F}_k^{\text{\textleaf}}$-free.
\end{theorem}

Note that \Cref{thm:B09T08} is a special case of \Cref{thm:BT12distance} for the case $k = 1$.

\section{Proof of General Theorem (\Cref{thm:hereditary-path-eccentricity})}\label{sec:main}

\lv{
For convenience, we restate our main theorem.

\mainThm*
}

\sv{\begin{proof}[Proof of \Cref{thm:hereditary-path-eccentricity}]}
\lv{\begin{proof}}
Fix an integer $k \ge 1$.
Let $G$ be a graph such that every connected induced subgraph $H$ of $G$ satisfies $\pe(H)< k$.
To show that $G$ is $\{S_{k},T_{k}\}$-free, it suffices to show that $\pe(S_k)\ge k$ and $\pe(T_k)\ge k$.
Let $H$ be a graph isomorphic to either $S_{k}$ or $T_{k}$ and let $S$ be any path in $H$.
We claim that $S$ has eccentricity at least $k$.
Suppose not.
Let $C$ be the set of vertices of degree $3$ in $H$.
Then, the graph $H-V(C)$ consists of three disjoint paths $P$, $Q$, and $R$, each with $k$ vertices.
Moreover, each of the paths $P$, $Q$, and $R$ contains a vertex at distance $k$ from $V(C)$.
Since the eccentricity of $S$ is less than $k$,
each of the paths $P$, $Q$, and $R$ contains a vertex of $S$.
Since $S$ has at most two endpoints, at most two of the paths $P$, $Q$, and $R$ contain an endpoint of $P$; hence, we may assume without loss of generality that $P$ contains an internal vertex of $S$ but no endpoint of $S$.
However, this implies that two distinct edges of $S$ connect $P$ with the rest of $H$, a contradiction.
This shows that $\pe(H)\ge k$, as claimed.

It remains to prove that if $G$ is $\{S_{k},T_{k}\}$-free, then every connected induced subgraph $H$ of $G$ has path eccentricity less than $k$, or, equivalently, that every connected $\{S_{k},T_{k}\}$-free graph has path eccentricity less than $k$.
For $k = 1$, the statement coincides with that of \Cref{thm:clawnetfree}.
Suppose now that $k\ge 2$.
Let $\mathcal{F} = \{S_1,T_1\}$ and let $\mathcal{D}$ be the class of all connected $\mathcal{F}$-free graphs.
Note that $\mathcal{F}_{k-1}^{\text{\textleaf}} = \{S_{k},T_{k}\}$.
Consequently, since $\mathcal{D}$ satisfies the hypothesis of \Cref{thm:BT12distance}, it follows that every $\{S_{k},T_{k}\}$-free graph $G$ is hereditarily $(k-1)$-dominated by $\mathcal{D}$.
Let $H$ be a connected induced subgraph of $G$.
Then, since $G$ is hereditarily $(k-1)$-dominated by $\mathcal{D}$, there exists a distance-$(k-1)$ dominating set $D$ in $H$ such that the induced subgraph $H[D]$ belongs to $\mathcal{D}$.
By \Cref{thm:clawnetfree}, the graph $H[D]$ admits a Hamiltonian path $P$.
It follows that $P$ is a path in $H$ with eccentricity less than $k$, implying that $\pe(H)< k$.
\end{proof}

\lv{
Our proof of \cref{cor:H-freeGeneral} will be based on the following.

\begin{lemma}\label{lem: IS both SkTk}
Let $k\geq 1$ be an integer and $H$ be a graph.
Then, $H$ is an induced subgraph of both $S_{k}$ and $T_{k}$ if and only if $H$ is an induced subgraph of $3P_{k}$ or $P_{2k+1}+P_{k-1}$.
\end{lemma}

\begin{proof}
Suppose first that $H$ is an induced subgraph of $3P_{k}$ or $P_{2k+1}+P_{k-1}$. 
Since $3P_{k}$ and $P_{2k+1}+P_{k-1}$ are induced subgraphs of both $S_{k}$ and $T_{k}$, it immediately follows that  $H$ is an induced subgraph of both $S_{k}$ and $T_{k}$.

Suppose now that $H$ is an induced subgraph of both $S_{k}$ and $T_{k}$.
If $H$ admits a vertex $u$ of degree $3$, then, since $H$ is an induced subgraph of $T_{k}$, two neighbors of $u$ are adjacent, forming a cycle.
However, this is impossible, since $H$ is also a subgraph of the acyclic graph $S_{k}$.
Thus $H$ has maximum degree at most $2$.
This implies that $H$ is an induced subgraph of either the graph obtained by removing the vertex of degree $3$ from $S_{k}$, which is isomorphic to $3P_k$, or the graph obtained by removing a neighbor of the vertex of degree $3$ from $S_{k}$, which is isomorphic to $P_{2k+1}+P_{k-1}$.
\end{proof}

\mainCor*
\begin{proof} 
Suppose first that $H$ is an induced subgraph of $3P_{k}$ or $P_{2k+1}+P_{k-1}$.
By \cref{lem: IS both SkTk}, $H$ is an induced subgraph of both $S_{k}$ and $T_{k}$.
This implies that every connected $H$-free graph $G$ is $\{S_{k},T_{k}\}$-free and, hence, satisfies $\pe(G)< k$ thanks to \Cref{thm:hereditary-path-eccentricity}.

Suppose now that $H$ is not an induced subgraph of $3P_{k}$ or $P_{2k+1}+P_{k-1}$.
Then, again by \cref{lem: IS both SkTk}, $H$ is not an induced subgraph of either $S_{k}$ or $T_{k}$. 
Let $R_{k}$ be a graph in $\{S_{k},T_{k}\}$ that does not admit $H$ as an induced subgraph. 
Then $R_{k}$ is a connected $H$-free graph and ${\pe(R_{k})\geq k}$ (by \Cref{thm:hereditary-path-eccentricity}).
\end{proof}

\section{Implications for the *-C1P Property (\cref{th:C1P})}\label{sec:C1P}

We start with the following lemma, which is an unpublished result of Bastide, Robinson, and the second author, obtained during the writing of \cite{BHR25}. 
We include the proof here for completeness.

\begin{lemma}\label{claim:C1P}
The graphs $S_2$ and $T_1$ do not have the *-C1P.
\end{lemma}
\begin{proof} 
    Since the argument is the same for both graphs, we use the same notations for their vertices.
    Let $a'$, $b'$, and $c'$ be the three leaves in $S_2$ (respectively $T_1$), $a$, $b$, and $c$ their respective neighbors and let $x$ the central vertex of $S_2$ (see \cref{fig:s2t1}).
    Suppose that $S_2$ (respectively $T_1$) has the *-C1P. 
    Then there exists an ordering $\sigma$ on the vertices such that for each vertex $u$, there exists $N_u\in \{N[u] ,N(u)\}$ (in particular, $N(u)\subseteq N_u \subseteq N[u]$), such that $N_u$ consists of consecutive vertices in the ordering $\sigma$. 
    Observe that in both graphs, the sets $N_a$, $N_b$, and $N_c$ are pairwise intersecting. 
    Since these sets correspond to consecutive sets in the ordering $\sigma$, one of those sets, w.l.o.g., $N_a$, is contained in the union of the others, here $N_b$ and $N_c$.
    However, $a'$ belongs to $N(a)\subseteq N_a$ and does not belong to $N[b]\cup N[c]$, and, hence, also does not belong to $N_b\cup N_c$, a contradiction.
\end{proof}
\begin{figure}
    \centering
    \includegraphics[width=0.7\linewidth]{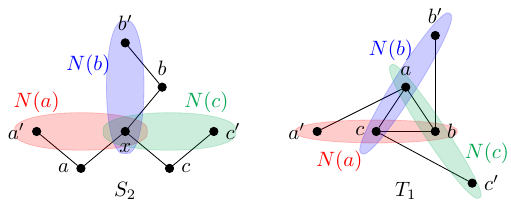}
    \caption{Notation for the vertices of $S_2$ and $T_1$ in \Cref{claim:C1P}, with the open neighborhoods of $a$, $b$, and~$c$.}
    \label{fig:s2t1}
\end{figure}

We restate our main theorem of this section, confirming a conjecture from~\cite{BHR25}.
\CoPthm*
\begin{proof}
It follows directly from the definition of \hbox{*-C1P} that if $G$ is a graph that has the *-C1P and $H$ is an induced subgraph of $G$, then $H$ also has the *-C1P.
Thus, by~\cref{claim:C1P}, graphs with *-C1P are $\{S_2,T_1\}$-free.
Therefore, they are also $\{S_2,T_2\}$-free and, hence, have path eccentricity at most 1 by \cref{thm:hereditary-path-eccentricity}.
\end{proof}

\ev{\section{Algorithmic Proofs}\label{sec:additional}}
\cv{\section{Algorithmic Proof (\cref{thm:longest-path})}\label{sec:additional}}

\ev{
We next prove our main result regarding $3P_2$-free graphs.
While the statement of the theorem is not algorithmic, the proof is, in the sense that it yields an efficient algorithm transforming a given longest path in a $3P_2$-free graph into a longest path that is dominating.
(Recall, however, that identifying a longest path in a given $3P_2$-free graph is an \textsf{NP}-hard problem.)
}

In this section, we use the following notation regarding paths in graphs.
Given a path $P$ with extremities $u,v$ and vertices $x,y$ adjacent respectively to $u,v$, we denote by $xPy$ the path obtained by adding the edges $xu$, $vy$ to $P$. 
Similarly, we denote by $xP$ and $Py$ the paths obtained by adding respectively only $xu$ or only $vy$ to $P$.
\ev{%
For any integer $k\ge 0$ and path $P$ in a graph $G$, we denote by $N_k[P]$ the set of vertices at distance at most $k$ from some vertex of $P$. 
For $k = 1$, we write $N[P]$ instead of $N_1[P]$.
We say that a path $P$ \emph{$k$-dominates} a vertex $x\in V(G)$ if $x\in N_k[P]$.
For $k = 1$, we simply say that $P$ \emph{dominates}~$x$.}
\cv{We say that a path $P$ \emph{dominates} a vertex $x\in V(G)$ if $x\in N[P]$, that is, if $x$ is in $P$ or in the neighborhood of $P$.
}

\longestThm*

\begin{proof}
  Let $P$ be a longest path in $G$.
  Suppose that $P$ is not dominating.
  In such a case, we show how to modify $P$ to obtain a longest path $P'$ that dominates more vertices than $P$.
  Let $u,v$ be the extremities of~$P$.

Since $P$ does not dominate $G$ and $G$ is connected, there exists a vertex $x$ at distance $2$ from $P$. 
Let $y$ be a neighbor of $x$ adjacent to some vertex of $P$, and $w$ be a neighbor of $y$ in $P$ that is closest to $u$ along $P$.
Observe that $y$ and $w$ have to exist, as $x$ is at distance $2$ from $P$.

Observe that neither $u$ nor $v$ has a neighbor outside of $P$, as $P$ is a longest path.
Hence, $w$ is an internal vertex of $P$, implying in particular that $u$, $v$, and $w$ are three distinct vertices of $P$.
Denote by $u'$, $v'$ the neighbor of $u$, respectively $v$ on $P$.
Observe that $u'\neq v$ and $v'\neq u$.
Let $P^*$ be the subgraph of $P$ obtained by deleting $u$ and $v$ from $P$.
In the case $y$ is adjacent to any of $u',v'$, we define $P'$ as $xyP^*v$ or $uP^*yx$. 
In both cases, $P'$ has a greater length than $P$, a contradiction.
Therefore, $x$ and $y$ are non-adjacent to any of $u$, $v$, $u'$, and $v'$. 
Moreover, as $y$ is adjacent to $P$, the vertices $u'$, $w$, and $v'$ are pairwise distinct, implying that $|V(P)|\geq 5$.

Let $P^1$ and $P^2$ be the two paths resulting from deleting $u$, $v$, and $w$ from $P^*$, containing respectively $u'$ and $v'$. 
Suppose that $u'v\in E(G)$, or $uv'\in E(G)$, or $uv\in E(G)$.
Then, the path $Q$ defined respectively by $xywP^1vP^2$, or $xywP^1uP^2$, or $xywP^1uvP^2$ is longer than $P$, a contradiction.
Therefore, $\{u'v,uv',uv\}\cap E(G) = \emptyset$.
Suppose that $u'$ is non-adjacent to $v'$. 
Then, the set $\{u,v,u',v'\}$ induces a $2P_2$ in $G$ and consequently $\{x,y,u,v,u',v'\}$ induces a $3P_2$, a contradiction.
Therefore, $u'v'\in E(G)$.  
Let $P'$ be the path $xywP^1P^2$. 
Note that the path $wP^1P^2$ has vertex set $V(P)\setminus\{u,v\}$, and, since neither $u$ nor $v$ has a neighbor outside of $P$, the path $wP^1P^2$ dominates as many vertices as $P$.
Hence, $P'$ has the same length as $P$, but it dominates more vertices.

The above argument shows that we can modify a given longest path $P$ at most $|V(G)|$ times until we obtain a longest path that is dominating.\cv{

It remains to discuss the running time in the case where we start the process above with some arbitrary (not necessarily the longest) path.
Then, in each step of the argument above, we modify the current path $P$ to a path $P'$ that is either longer than $P$ or has the same length but dominates more vertices.
Therefore, we run the standard BFS algorithm, checking whether path $P$ is already dominating.
If not, it finds a vertex $x$ at distance $2$ from $P$ together with a connecting path back to $P$, and we proceed with the algorithm finding $P'$ and repeat.
BFS runs in time $\Oh(n+m)$.
Our proof guarantees at most $n^2$ steps, as in each step we either increase the size of the dominated set by at least $1$, or we do not decrease it, but we increase the length of path $P$.}
\end{proof}

\ev{
Before we proceed with the proof of \Cref{thm:3P2}, we will need the following lemma from \cite{DDHMPT24} regarding a sufficient condition for the presence of an induced subgraph isomorphic to a graph from $\mathcal{S}$, $\mathcal{T}$, or $\mathcal{M}$ (recall the definitions on p.~\pageref{fig:S-T-M} and~\cref{fig:S-T-M}). 

\begin{lemma}[Dallard et al.~{\cite[Lemma 1]{DDHMPT24}}]\label{lem:STM}
  Let $G$ be a graph and $I$ be an independent set in $G$ with $|I| = 3$.
  If there exists a component $C$ of $G \setminus I$ such that $I \subseteq N(C)$, then there exists an induced subgraph $H$ of $G[N[C]]$ such that $H \in \mathcal{S} \cup \mathcal{T} \cup \mathcal{M}$ and $I$ is exactly the set of extremities of $H$.
\end{lemma}
Moreover, from the proof of the above lemma (see Lemma 6.1,\cite{BDDHMPT25-arXiv}), we can deduce that $H$ can be found in  $\mathcal{O}(n+m)$ time.

\thmAlg*

\begin{proof}
Let $P$ be an arbitrary path in $G$ and $u, v$ be the extremities of $P$.

The algorithm proceeds recursively as follows.
While $N_{k-1}[P]\neq V(G)$, we construct a path $P'$ such that:
\begin{itemize}
  \item [(i)] $N_{k-1}[P]\subsetneq N_{k-1}[P']$ (corresponding to \cref{cl:extendingNeigh}), or\phantomsection\label{c:ii}
  \item[(ii)] $N_{k-1}[P]\subseteq N_{k-1}[P']$ and $|V(P')| < |V(P)|$  (corresponding to \cref{cl:shorteningP}). \phantomsection\label{c:i}
\end{itemize}

After at most $n$ updates of type \hyperref[c:ii]{(i)}, $N_{k-1}[P]= V(G)$, that is, a path of with eccentricity less than $k$ in $G$ is obtained.
Before an update of type \hyperref[c:ii]{(i)}, we create a nice setting for such update by applying updates of type \hyperref[c:i]{(ii)} as many times as possible (which is at most $n$ times).
If at some point it is not possible to execute an update of type \hyperref[c:ii]{(i)}, we exhibit an induced subgraph isomorphic to either $S_k$, $T_k$, or a graph in $\mathcal{M}_k$.
In what follows, we show that one update takes $\Oh(n+m)$.
This will imply that the running time of the algorithm is $\Oh(n^2(n+m))$.

\begin{claim}\label{cl:shorteningP}
If $u\neq v$, 
then at least one of the following holds.
\begin{itemize}
  \item $N_{k-1}[P-u]=N_{k-1}[P]$,
  \item $N_{k-1}[P-v]=N_{k-1}[P]$, or
  \item there exist two distinct vertices $u'$ and $v'$ in $G$ and two paths $P_u$ and $P_v$ with extremities $u,u'$ and $v,v'$ respectively, such that $V(P_u)\cap V(P_v) = \emptyset$, $V(P_u)\cap V(P) = \{u\}$, 
  $V(P_v)\cap V(P) = \{v\}$, and the sets $V(P_u)$ and $V(P_v)$ each induce a $P_k$ in $G$.  
   Moreover, the only edge between $V(P_u)\setminus \{u\}$ (respectively $V(P_v)\setminus \{v\}$) and $V(P)$ is the edge between $u$ (respectively $v$) and its neighbor along $P_u$ (respectively~$P_v$). 
\end{itemize}  
Distinguishing which of these cases occurs and, in the case of the last outcome, a certificate for it can be found in time $\Oh(n+m)$.
\end{claim}
\begin{proof}[Proof of \cref{cl:shorteningP}]
  Suppose that $u\neq v$.
  If $N_{k-1}[P-u]\subsetneq N_{k-1}[P]$, then there exists a vertex $u'$ at distance $k-1$ from $u$ such that $u'$ is at distance $k$ from $P-u$. 
  Let $P_u$ be a path witnessing distance $k-1$ between $u'$ and $u$.
  Then $P_u$ induces a $P_k$ in $G$ and is internally disjoint from $P$ since $u'$ is at distance $k$ from $P-u$. 
  Moreover, the only edge between $V(P_u)\setminus \{u\}$ and $V(P)$ is the edge between $u$ and its neighbor along $P_u$.
  If $N_{k-1}[P-v]\subsetneq N_{k-1}[P]$ also holds, we construct similarly $v'$ and $P_v$ with analogous properties.

  Suppose that $P_u$ and $P_v$ intersect in some vertex $x$.
  Without loss of generality, the distance from $u'$ to $x$ along $P_u$ is at most the distance from $v'$ to $x$ along $P_v$. 
  Then, following first $P_u$ between $u'$ and $x$ and then following $P_v$ between $x$ and $v$, we get a path from $u'$ to $v$ at most as long as $P_v$, which has length $k-1$, thus $u'$ is at distance at most $k-1$ from $v\in V(P-u)$, a contradiction.

  The time complexity follows, as we first check whether one of the first two outcomes holds, comparing outcomes of the BFS algorithms starting with $P$, $P-u$, and $P-v$, respectively.
  If all those outcomes are different, they provide us with $u'$ and $v'$ and the respective paths $P_u$ and $P_v$.
\renewcommand\qedsymbol{$\diamondsuit$}
\end{proof}

Assume that $u\neq v$.
Note that if the first or the second case of \cref{cl:shorteningP} holds, then we can replace $P$ by either $P-u$ or $P-v$, and get a path strictly smaller with the same neighborhood at distance $k-1$.
Repeating this at most $|V(P)|<n$ times, we obtain $P'$ such that $N_{k-1}[P']=N_{k-1}[P]$ and $|V(P')|\leq |V(P)|$, and either $P'$ is reduced to a single vertex, or we get the existence of $u',v',P_u,P_v$ as defined in \cref{cl:shorteningP} (with $P'$ in place of $P$). 
We then replace $P$ by $P'$.

We can now assume that either $P$ is reduced to a single vertex, or there exist $u',v',P_u,P_v$ as defined in \cref{cl:shorteningP}.
We can also assume that $N_{k-1}[P]\neq V(G)$, since otherwise $P$ is a path of eccentricity less than $k$ and we can return it.
Thus, there exists a vertex $x$ at distance $k$ from $P$.
Therefore, there is an induced $P_k$ in $G$, denoted $P_x$, with extremities $x$ and $y$ such that $y$ is the only vertex of $P_x$ adjacent to some vertex of $P$.

\begin{claim}\label{cl:extendingNeigh}
Let  \[
S :=
\begin{cases}
  V(P_x)\cup V(P_u)\cup V(P_v), \hfill\text{~~~ if $|V(P)|\neq 1$.}\\
  V(P_x),  \hfill \text{~~~otherwise.}
\end{cases}
\]
  Then, either $S$ induces a $3P_k$ or we can construct a path $P'$ in $S \cup V(P)$ such that $N_{k-1}[P]\subsetneq N_{k-1}[P']$.
A certificate for one of the outcomes can be found in time $\Oh(n+m)$.
\end{claim}
\begin{proof}[Proof of \cref{cl:extendingNeigh}]
  If $P$ is reduced to a single vertex, then, by the construction of $P_x$, the set $V(P)\cup V(P_x)$ induces a path that contains $V(P)\cup \{x\}$, implying that $N_{k-1}[P]\subsetneq N_{k-1}[P']$, since $x \notin N_{k-1}[P]$, so the claim holds.

  Suppose now that $u \neq v$ (thus $u',v',P_u$, and $P_v$ are defined as in \cref{cl:shorteningP}).
  Recall that each path $P_u, P_v, P_x$ taken individually is an induced $P_k$ in $G$.
  If $V(P_x)\cup V(P_u)\cup V(P_v)$ induces a $3P_k$ then return this outcome.
  Otherwise, this means that either they share some vertices or that there exists an edge with endpoints in two of those three paths. 
  In each case, we will construct a path $P'$ in $V(P_x)\cup V(P_u)\cup V(P_v)\cup V(P)$ that contains   $V(P)\cup \{x\}$, implying that $N_{k-1}[P]\subsetneq N_{k-1}[P']$, since $x \notin N_{k-1}[P]$.

  Suppose first that $P_x$, $P_u$, and $P_v$ are not vertex-disjoint.
  Since by \cref{cl:shorteningP}, $P_u$ and $P_v$ are vertex-disjoint, $V(P_x)$ intersects $V(P_u)\cup V(P_v)$. 
  Let $y'$ be the vertex of $P_x$ closest to $x$ along 
  $P_x$ that belongs to $V(P_u)\cup V(P_v)$. 
  By symmetry, we may assume that $y'\in V(P_u)$, and we denote $P_x'$ and $P_u'$ the respective subpaths of $P_x$ with extremities $x,y'$ and of $P_u$ with extremities $y',u$. Then we construct our wanted $P'$ as the concatenation of $P_x'$, $P_u'$, and $P$. We can thus assume that $V(P_x)$, $V(P_u)$, and $V(P_v)$ are pairwise disjoint.

  If there is an edge $e$ between a vertex $y'\in V(P_x)$ and a vertex $u''\in V(P_u)$, we define $P_x'$ and $P_u'$ as the respective subpaths of $P_x$ with extremities $x,y'$ and of $P_u$ with extremities $u'',u$. Then we construct our wanted $P'$ as the concatenation of $P_x'$, $e$, $P_u'$, and $P$.
  The case where $e$ is between a vertex of $P_x$ and $P_v$ is obtained by symmetry. We can now assume that there is no edge between $V(P_x)$ and $V(P_u)\cup V(P_v)$.

  Finally, suppose there is an edge $e$ between a vertex $u''\in V(P_u)$ and a vertex $v''\in V(P_v)$, we define $P_u'$ and $P_v'$ as the respective subpaths of $P_u$ with extremities $u'',u$ and of $P_v$ with extremities $v,v''$. Furthermore, let $w$ be a neighbor of $y$ in $P$. Since there is no edge between $V(P_u)$ and $V(P_x)$, we have $w\neq u$. Thus, there exists $w'$, the neighbor of $w$ in $P$ toward $u$. Let $Q$ and $Q'$ be two subpaths of $P$ with extremities $w,v$ and $u,w'$ respectively; note that $V(Q)\cup V(Q')=V(P)$.
  Then we construct our wanted $P'$ as the concatenation of $P_x$, $yw$, $Q$, $P_v'$, $e$, $P_u'$, and $Q'$.

  The time complexity follows, as any edge or a shared vertex can be discovered within $S$ by scanning the closed neighborhood of $P_x$ and $P_u$, in which case we can output $P'$.
\renewcommand\qedsymbol{$\diamondsuit$}
\end{proof}

Now suppose that \cref{cl:extendingNeigh} returned the outcome that $S=V(P_x)\cup V(P_u)\cup V(P_v)$ induces a $3P_k$. 
Let $G'=G[V(P) \cup \{y\}]$.
Then $I=\{u,v,y\}$ is an independent set in $G'$, since $S$ induces a $3P_k$.  
Let $C$ be the subgraph of $G'$ induced by $V(P)\setminus I$. Then $I\subset N(C)$ and by \cref{lem:STM} the graph $G'[N[C]]$ contains an induced subgraph $H' \in \mathcal{S} \cup \mathcal{T} \cup \mathcal{M}$ with extremities $\{u,v,y\}$.
Thus, $H=G'[V(H')\cup S]$ either
\begin{itemize}
  \item belongs to $\mathcal{M}_k$ if $H'\in \mathcal{M}$,
  \item contains an $S_k$ as an induced subgraph if $H'\in \mathcal{S}$,
  \item or contains a $T_k$ as an induced subgraph if $H'\in \mathcal{T}$.
\end{itemize}

By applying \cref{cl:shorteningP} and \cref{cl:extendingNeigh}, we get one of the outcomes \hyperref[c:ii]{(i)} or \hyperref[c:i]{(ii)} (or we exhibit the wanted induced subgraph $H$) and, as discussed, the number of such steps is at most $n^2$.
The running time of each individual step is bounded by $\mathcal{O}(m+n)$, as stated in the respective claims.
The claimed time complexity follows.
\end{proof}
}

\section{Concluding Remarks and Open Questions}\label{sec:conclusion}

Our main result, \cref{thm:hereditary-path-eccentricity}, shows that for every integer $k \ge 1$, each connected $\{S_{k},T_{k}\}$-free graph has path eccentricity less than $k$.
This raises the question regarding the complexity of computing the path eccentricity of such a graph.
While the problem is clearly polynomial for $k = 1$, it turns out that the problem is \textsf{NP}-hard for all $k\ge 2$.
This follows from the fact that the problem of determining whether a given graph has a Hamiltonian path is \textsf{NP}-complete in the class of split graphs~\cite{MR1405031}, which are known to be $2P_2$-free and, hence, $\{S_{2},T_{2}\}$-free.
This, however, says nothing about the complexity of computing a path witnessing the upper bound on path eccentricity guaranteed by \cref{thm:hereditary-path-eccentricity}.

\begin{question}\label{question-algo}
Is there an algorithmic version of \Cref{thm:hereditary-path-eccentricity}?
More precisely: for a fixed integer $k \ge 1$ and a connected $\{S_{k},T_{k}\}$-free graph $G$, can we find in polynomial time a path in $G$ with eccentricity less than $k$?
\end{question}

While \Cref{question-algo} is known to have an affirmative answer for $k = 1$ (see~\cite{Shepherd1991,Brandstadt2000}), it is, to the best of our knowledge, open for all $k\ge 2$.
\Cref{thm:3P2} provides a partial answer.
\cv{In fact, even a weaker version of \cref{question-algo}, finding a path of eccentricity less than $k$ in cases covered by \cref{cor:H-freeGeneral}, is open, except for the case of $3P_2$-free graphs, which is covered by \cref{thm:longest-path}.}

Another result of this paper (\cref{thm:longest-path}) is that every connected $3P_{2}$-free graph $G$ has a longest path that is dominating.
This result, along with the observation that there exist graphs that are not $3P_2$-free such that every connected induced subgraph has a longest path that is dominating (consider, for instance, paths and cycles), raises the question about a characterization of such graphs.

\begin{question}
What are the graphs $G$ such that every connected induced subgraph of $G$ has a longest path that is dominating?
\end{question}

Camby and Schaudt showed in~\cite{Camby2016} that for every integer $k \geq 4$, every connected $P_k$-free graph $G$, and every \textsl{minimum} connected dominating set $S$ in~$G$, the subgraph of $G$ induced by $S$ is either $P_{k-2}$-free or isomorphic to~$P_{k-2}$.
Can some similar properties regarding minimum connected dominating sets (at distance) be proved for $\{S_{k},T_{k}\}$-free graphs?
In particular:

\begin{question}\label{question-minimum}
Is it true that for every integer $k\ge 2$, every connected $\{S_{k},T_{k}\}$-free graph has a minimum connected distance-$(k-1)$ dominating set inducing a $\{S_1,T_1\}$-free graph?
\end{question}
}%

\lv{
\subsection*{Acknowledgements}

The authors are grateful to Ulrich Pferschy for pointing out reference~\cite{Antony2024}, which motivated the present research.
}

\bibliographystyle{alphaurl}
\bibliography{biblio}
\newpage
\appendix 
\appendixText

\end{document}